\documentclass[twoside,11pt]{amsart}

\usepackage{a4wide}
\usepackage{latexsym}
\usepackage{mathrsfs}
\usepackage[T1]{fontenc}
\usepackage{enumitem}
\usepackage{mathrsfs}

\usepackage{amssymb}
\usepackage{latexsym}
\usepackage{amsmath}
\usepackage{fancyhdr}
\usepackage{bbm}
\usepackage{setspace}
\usepackage{mathabx}

\numberwithin{equation}{section}

\newtheorem{theorem}{Theorem}[section] 
\newtheorem{lemma}[theorem]{Lemma}     

\newtheorem{proposition}[theorem]{Proposition}
\theoremstyle{definition}
    
\newtheorem{conjecture}[theorem]{Conjecture}

\newcommand {\C}{\mathbb C}

\newcommand {\N}{\mathbb N}

\newcommand {\B}{\mathcal B}

\newcommand{\F}{\mathbb{F}}

\newcommand{\supp}{{\rm supp\,}}

\newcommand{\eps}{\varepsilon}
\newcommand{\spn}{{\rm span \,}}

\newcommand{\im}{\operatorname{im}}

\newcommand{\codim}{\operatorname{codim}}

\title[On The Dales--{\.Z}elazko Conjecture]{ On The Dales--{\.Z}elazko Conjecture for Beurling Algebras on Discrete Groups}



\begin{document}

	\author[J.\ T.\  White]{Jared T.\ White}
	\address{
		Jared T. White, School of Mathematics and Statistics, The Open University, Walton Hall, Milton Keynes MK7 6AA, United Kingdom}
	\email{jared.white@open.ac.uk}

	\date{2023}
	
	\maketitle
	
		\begin{center}
	\textit{The Open University}
\end{center}
	
\begin{abstract}
	Let $G$ be a group which is either virtually soluble or virtually free, and let $\omega$ be a weight on $G$. We prove that, if $G$ is infinite, then there is some maximal left ideal of finite codimension in the Beurling algebra $\ell^1(G, \omega)$ which fails to be (algebraically) finitely generated. This implies that a conjecture of Dales and {\.Z}elazko holds for these Banach algebras. We then go on to give examples of weighted groups for which this property fails in a strong way. For instance we describe a Beurling algebra on an infinite group in which every closed left ideal of finite codimension is finitely generated, and which has many such ideals in the sense of being residually finite dimensional.
These examples seem to be hard cases for proving Dales and {\.Z}elazko's conjecture.
\end{abstract}

\section{Introduction}
In this article we explore some connections between the properties of a discrete group $G$ and the properties of the left ideals of its Beurling algebras $\ell^1(G, \omega)$.
 We shall show that the existence or not of a weight $\omega$ on $G$ such that $\ell^1(G, \omega)$ has a finitely-generated maximal left ideal of finite codimension depends on properties of the group $G$. One of our key methods will be based on weights defined in terms of word length with respect to a given generating set for the group.

We begin with some background and motivation. Let $A$ be an infinite dimensional unital Banach algebra, and let $I$ be a closed left ideal of $A$. There is a natural tension between finite generation of $I$ and various ways that $I$ could be ``large'', such as being maximal or having finite codimension. It is this idea that motivates the following conjecture of Dales and {\.Z}elazko \cite{DZ}:

\begin{conjecture}[Dales--{\.Z}elazko]
 Let $A$ be an infinite dimensional unital Banach algebra. Then $A$ has a maximal left ideal which is not finitely generated. 
\end{conjecture}

The conjecture has since attracted some attention, and is known to be true for various special classes of Banach algebras, for example commutative Banach algebras \cite{FT}, C*-algebras \cite{BK}, $\ell^1(G)$ for a group $G$ \cite{W1}, and the algebra of bounded linear operators on a large class of Banach spaces \cite{DKKKL}, including the reflexive Banach spaces \cite[Corollary 2.2.7]{W0}. At the time of writing no counterexample is known. 
 
 In \cite{W1} the author showed that the Dales--{\.Z}elazko Conjecture holds for Beurling algebras on discrete groups, provided that the weight is sufficiently well-behaved. In the present work we instead consider arbitrary weights but a restricted class of groups.  Our first result is the following.

\begin{theorem}		\label{2.2}
	Let $G$ be an infinite group which is either 
	\begin{enumerate}
		\item[\rm (i)] a virtually soluble group, or
		\item[\rm (ii)] a virtually free group,
	\end{enumerate}
	and let $\omega$ be an arbitrary weight on $G$. Then $\ell^1(G, \omega)$ contains a maximal left ideal of finite codimension which is not finitely generated. In particular $\ell^1(G, \omega)$ satisfies the Dales--\.Zelazko conjecture.
\end{theorem}

One innovation of the present work on \cite{W1} is a lemma which allows us to go up finite-index group extensions. Then the fact that Dales and {\.Z}elazko's conjecture is known to hold in the commutative case means that it holds for any Beurling algebra on an abelian group, regardless of the weight. Combining these two ideas will prove the more general Proposition \ref{2.1b}, which will imply Theorem \ref{2.2}.

The results in \cite{W1}, as well as Theorem \ref{2.2}, establish the Dales--{\.Z}elazko Conjecture by finding examples of finite-codimension left ideals. These left ideals are usually easier to work with than those of infinite codimension. However, as we shall see, a Beurling algebra need not have a finite-codimension maximal left ideal that fails to be finitely generated.
Indeed, our next result shows that Beurling algebras can have starkly different behaviour to that displayed in Theorem \ref{2.2}. First of all, it can happen that there is only one finite-codimension maximal left ideal.

\begin{theorem}		\label{2.6}
	Let $G$ be a finitely-generated, infinite simple group. Fix a finite, symmetric generating set $X$ for $G$, and let $\omega(t) = 2^{|t|_X} \ (t \in G)$. Then the only proper, finite-codimension left ideal of $\ell^1(G, \omega)$ is the augmentation ideal $\ell^1_0(G, \omega)$, and it is finitely generated.
\end{theorem}

On the other hand, it can happen that the Beurling algebra has ``many'' finite-codimension maximal left ideals, all of which are  finitely generated. We do not know of any other example of an infinite dimensional unital Banach algebra satisfying (i) and (ii).

\begin{theorem}		\label{2.7}
	Let $G$ be a finitely-generated, just infinite, residually finite group, which is not a linear group. Fix a finite, symmetric generating set $X$ for $G$, and set $\omega(t)  = 2^{|t|_X} \ (t \in G)$. Then $\ell^1(G, \omega)$ has the following properties:
	\begin{enumerate}
		\item[\rm (i)] the intersection of the finite-codimension, maximal left ideals of $\ell^1(G,\omega)$ is $\{ 0 \}$;
		\item[\rm (ii)] every finite-codimension left ideal is finitely generated.
	\end{enumerate}
\end{theorem}

 Examples of groups satisfying this theorem include finitely-generated, just infinite branch groups, such as the Grigorchuk group and the Gupta--Sidki $p$-groups; it follows from \cite[Corollary 7]{A06} that branch groups cannot be linear. For background on branch groups and their relationship to the theory of just infinite groups we refer the reader to \cite{BGS}.

Note that in part (ii) of Theorem \ref{2.7} we do not assume that the left ideals are closed. However, as we shall remark after the proof, every finite codimension left ideal of this Beurling algebra is in fact closed. Likewise, in Theorem \ref{2.6} the augmentation ideal is the only proper finite-codimension left ideal, closed or otherwise.

These two theorems entail that, if the Dales--\.Zelazko conjecture is true for Beurling algebras, then in order to prove it one needs to look beyond finite-codimension left ideals. We have been unable to resolve the conjecture for these examples.

\section{Notation and Terminology}

We now fix our notation and terminology. Let $A$ be a unital Banach algebra, and let $I$ be a left ideal of $A$. We say that $I$ is \textit{finitely generated} if there exists $n \in \N$ and $a_1, \ldots, a_n \in I$ such that $I = Aa_1 + \cdots + Aa_n$. We shall usually study left ideals that are closed, but note that we do not take any closure on the right hand side in our definition of finite generation. Recall that a maximal left ideal in a unital Banach algebra is automatically closed, and that a codimension-1 left ideal is always maximal, although in general a finite-codimension left ideal need not be closed.

By a \textit{weight} on a group $G$ we mean a function $\omega \colon G \to [1, \infty)$ such that $\omega(st) \leq \omega(s)\omega(t) \ (s,t \in G)$ and $\omega(e) = 1$. We set
$$\ell^1(G, \omega) = \left\{f \colon G \to \C : \sum_{t \in G} |f(t)| \omega(t) < \infty \right\},$$
which is a Banach algebra with multiplication given by convolution, which we write as `$*$', and the norm given by $\|f\|_\omega = \sum_{t \in G} |f(t)| \omega(t) \ (f \in \ell^1(G, \omega)).$ Banach algebras of this form are called \textit{Beurling algebras}. We write $\C G$ for the set of finitely-supported, complex-valued functions of $G$, which is a dense subset of every Beurling algebra on $G$. Given $t \in G$ we write $\delta_t$ for the point-mass at $t$, and given $f \in \ell^1(G, \omega)$ we write $f^t := \delta_{t}*f*\delta_{t^{-1}}$.


Given a group $G$, a weight $\omega$ on $G$, and a normal subgroup $N$, we write $\widetilde{\omega}$ for the weight induced on the quotient $G/N$ defined by 
\begin{equation}		\label{eq5}
\widetilde{\omega}(Nt) = \inf\{ \omega(st) : s \in N \} \qquad (t \in G).
\end{equation}
 By \cite[Theorem 3.7.13]{RS} the quotient map $q \colon G \rightarrow G/N$ extends to a surjective bounded algebra homomorphism $q \colon \ell^1(G, \omega) \to \ell^1(G/N , \widetilde{\omega})$, whose kernel is given by
$$J_\omega(G,N) := \left\{ f \in \ell^1(G,\omega) : \sum_{s \in Nt} f(s) = 0 \ (t \in G) \right\}.$$
In fact $\ell^1(G,\omega)/J_\omega(G,N)$ and 
$\ell^1(G/N, \widetilde{\omega})$ are isometrically isomorphic as Banach algebras.
We define the \textit{augmentation ideal of $\ell^1(G,\omega)$} to be 
$$\ell^1_0(G, \omega) := J_\omega(G,G) = \left\{f \in \ell^1(G, \omega) : \sum_{t \in G} f(t)  = 0 \right\}.$$

Let $G$ be a group with finite generating set $X$. We write $|t|_X$ for the word length of $t \in G$ with respect to $X$, and, given $r \in \N$, we write
$$B_r^X = \{ t \in G: |t|_X \leq r \}.$$
We also write $\dot{B}_r^X : = B_r^X \setminus \{ e \}$. A generating set $X$ is said to be \textit{symmetric} if, for all $x \in X$, we have $x^{-1} \in X$. We shall often insist that our generating sets are symmetric, both to ease notation and to apply the results of \cite{W1}.
Word length gives us an important way to construct weights on the group $G$. Given $c>1$ we can define a weight $\omega$ on $G$ by 
$$\omega(t) = c^{|t|_X} \qquad (t \in G).$$
Such a weight is sometimes called a \textit{radial exponential weight}. Weights of this form  with $c=2$ appear in two of our main results discussed above; the choice $c=2$ has merely been made for concreteness, and the analogous results also hold for arbitrary $c>1$.

Let $G$ be a group, and let $H$ be a subgroup. By a \textit{right/left transversal for $H$ in $G$} we mean a set of right/left coset representatives respectively. When $H$ is a normal subgroup every left transversal is a right transversal and vice versa so we refer simply to a transversal for $H$ in $G$.

We write $\F_n$ for the free group on the set $\{a_1, \ldots, a_n \}$, where $a_1, \ldots, a_n$ are just $n$ arbitrary symbols. We say that a group $G$ is a \textit{linear group} if it is isomorphic to a subgroup of $GL_n(\mathbb{K})$, the group of invertible $n$-by-$n$ matrices over a field $\mathbb{K}$. We say that an infinite group $G$ is \textit{just infinite} if all of its non-trivial quotients are finite.

\section{The Proofs}
\noindent
Our first task is to prove Theorem \ref{2.2}. We begin by proving a lemma that allows us to use our understanding of a finite-index subgroup to study the ambient group.

\begin{lemma}		\label{2.1}
Let $G$ be a group, and let $\omega$ be a weight on $G$. Let $H$ be a finite-index, normal subgroup of $G$, and suppose that $\ell^1(H, \omega\vert_H)$ has a closed finite-codimension left ideal which is not finitely generated. Then so does $\ell^1(G, \omega)$. 
\end{lemma}

\begin{proof}
Let $I$ be a finite-codimension left ideal in $\ell^1(H, \omega\vert_H)$ which fails to be finitely generated. We consider $\ell^1(H, \omega\vert_H)$ as a subalgebra of $\ell^1(G, \omega)$. Let $n = [G : H]$, and let $t_1, \ldots, t_n$ be a transversal for $H$ in $G$. Define
$$J = \left\{ \delta_{t_1}*f^{(1)} + \cdots + \delta_{t_n}*f^{(n)} : f^{(i)} \in I \ (i=1, \ldots, n) \right\}.$$
Note that $J$ is the $\ell^1$-direct sum of $\delta_{t_i}*I$ for $i=1, \ldots, n$, so it is closed. 
We show that $J$ is a left ideal of $\ell^1(G,\omega)$. Let $u \in G$ and for each $i = 1, \ldots, n$ write $ut_i = t_{j(i)}v_i$, where $v_i \in H \ (i = 1, \ldots, n)$, and $j$ is a permutation. Then given $f^{(i)} \in I \ (i=1, \ldots, n)$ we have 
$$\delta_u * \left(\sum_{i=1}^n \delta_{t_i}*f^{(i)} \right) = \sum_{i=1}^n \delta_{t_{j(i)}}*\delta_{v_i}*f^{(i)} \in J,$$
since each $\delta_{v_i}*f^{(i)} \in I$. It follows that $J$ is a left ideal in $\ell^1(G, \omega)$.

We claim that $\codim J = [G:H] \codim I,$ which is finite. Indeed, write $k = \codim I$, and suppose that $g_1, \ldots, g_k \in \ell^1(H, \omega \vert_H)$ have the property that $\{ g_1+I, \ldots, g_k+I \}$ is a basis for $\ell^1(H, \omega \vert_H)/I$. Then it is easily checked that $\{ \delta_{t_i}*g_j +J : i = 1, \ldots, n, \ j=1, \ldots, k \}$ is a basis for $\ell^1(G)/J$. 

Next we show that $J$ is not finitely generated.
Assume towards a contradiction that 
$$J = \ell^1(G, \omega)*f_1 + \cdots + \ell^1(G, \omega)*f_m,$$
 for some $m \in \N$, and some $f_1, \ldots, f_m \in J$. Then each $f_i$ may be written as 
$$f_i = \delta_{t_1}*f_i^{(1)} + \cdots + \delta_{t_n}*f_i^{(n)},$$
 for some $f_i^{(1)}, \ldots, f_i^{(n)} \in I$. 
Let $g \in I$ be arbitrary. Then in particular $g \in J$, so it may be written as $g = h_1*f_1 + \cdots + h_m * f_m$, for some $h_1, \ldots, h_m \in \ell^1(G, \omega)$.
 Again, we may write each $h_i$ as 
 $h_i = \delta_{t_1}*h_i^{(1)} + \cdots + \delta_{t_n}*h_i^{(n)}$, for some functions $h_i^{(1)}, \ldots, h_i^{(n)}  \in \ell^1(G, \omega)$ supported on $H$. We have
 $$h_i*f_i = \sum_{j, \, k=1}^n \delta_{t_j}*h_i^{(j)}*\delta_{t_k}*f_i^{(k)} 
= \sum_{j, \, k =1}^n \delta_{t_jt_k}*(h_i^{(j)})^{t_k^{-1}}*f_i^{(k)} 
= \sum_{j, \, k =1}^n \delta_{t_jt_k}*\phi_{ijk},$$
where $\phi_{ijk} := (h_i^{(j)})^{t_k^{-1}}*f_i^{(k)}$.
Since $H$ is normal, for all $i,j,k$ we have $\supp \phi_{ijk}  \subset H$, so that 
$\supp \left\{ \delta_{t_jt_k}*\phi_{ijk} \right\} \subset t_jt_k H$. 
 Since $g$ is supported on $H$ it must be equal to the sum of those terms $\delta_{t_jt_k}*\phi_{ijk}$ which are supported on $H$. For each $k = 1, \ldots, n$, write $k' \in \{ 1, \ldots, n \}$ for the natural number satisfying $t_k^{-1}H = t_{k'}H$.  We see that 
 $$g = g\vert_H = \sum_{k =1}^n \sum_{i=1}^m \delta_{t_{k'}t_k}* (h_i^{(k')})^{t_k^{-1}}*f_i^{(k)} .$$
 Hence, since $g$ was arbitrary, $I$ is generated by $\left\{ f_i^{(k)} : i=1, \ldots, m, \ k=1, \ldots, n \right\}$. This contradiction concludes the proof.
\end{proof}

In the proof of the next lemma $\B(E)$ denotes the Banach algebra of bounded linear operators on a Banach space $E$.

\begin{lemma}		\label{2.1a}
Let $A$ be a Banach algebra with a finite-codimension left ideal which is not finitely generated. Then:
	\begin{enumerate}
		\item[\rm (i)] $A$ contains a finite-codimension two-sided ideal which is not finitely generated as a left ideal;
		\item[\rm (ii)] if $A$ is also unital, then $A$ contains a finite-codimension maximal left ideal which is not finitely generated.
	\end{enumerate}
\end{lemma}

\begin{proof}
(i) Let $I$ be a left ideal of finite codimension in $A$ which is not finitely generated. Let $J$ be the kernel of the map $ \theta \colon I \rightarrow \B(A/I)$ given by $a \mapsto (x + I \mapsto ax+I)$. Then 
$$J = \{ a \in I : ax \in I \text{ for all } x \in A\}$$
and is a two-sided ideal contained in $I$ (in fact the largest one). The ideal $J$ has finite codimension in $I$, since $\dim(I/J) \leq \dim \im \theta \leq \dim \B(A/I) < \infty,$ and hence also has finite codimension in $A$. In particular we can write $I = J + \spn\{a_1, \ldots, a_n \}$, for some $n \in \N$, and some $a_1, \ldots, a_n \in I$. If $J$ were finitely generated as a left ideal, then taking a finite generating set for $J$ together with $a_1, \ldots, a_n$ would give a finite generating set for $I$. Hence $J$ cannot be finitely generated.

(ii) This follows from part (i) and \cite[Lemma 2.9]{DZ}.
\end{proof}

We are now ready to prove the first of our main results.

\begin{proposition}		\label{2.1b}
Let $G$ be an infinite group with a finite-index normal subgroup $H$ which has an infinite abelian quotient. Let $\omega$ be a weight on $G$. Then $\ell^1(G, \omega)$ has a finite-codimension maximal left ideal that is not finitely generated. In particular $\ell^1(G, \omega)$ satisfies the Dales--{\.Z}elazko Conjecture.
\end{proposition}

\begin{proof}
Since by \cite{DZ} the Dales--{\.Z}elazko Conjecture holds in the commutative case, the Banach algebra 
$\ell^1\left(H/H', \widetilde{\omega|_H} \right)$ has a codimension-1 ideal which fails to be finitely generated. By considering the preimage of this ideal under the quotient map $\ell^1(H, \omega|_H) \rightarrow \ell^1(H/H', \widetilde{\omega|_H})$ we see that $\ell^1(H, \omega|_H)$ has a codimension-1 ideal which is not finitely generated.  It follows from Lemma \ref{2.1} that $\ell^1(G, \omega)$ has a finite-codimension left ideal which fails to be finitely generated. Hence, by Lemma \ref{2.1a}(ii), $\ell^1(G, \omega)$ has a maximal left ideal of finite codimension which is not finitely generated, as required.
\end{proof}

\begin{proof}[Proof of Theorem \ref{2.2}]
(i) Suppose that $G$ is virtually soluble. Then $G$ has a finite-index soluble subgroup $S$ that we may take to be normal. Set $k = \min \{ i \in \N : [ S^{(i)} : S^{(i+1)}] = \infty \}$ and observe that $H := S^{(k)}$ is a finite-index subgroup of $G$ with an infinite abelian quotient. Moreover, since $H$ is characteristic in $S$ it is normal in $G$, so we we may apply Proposition \ref{2.1b}.

(ii) Suppose that $G$ is virtually free. Then $G$ contains a finite-index free subgroup $H$, which we may assume is normal. The result now follows from Proposition \ref{2.1b}.
\end{proof}

We now move on to our theorem about simple groups.

\begin{proof}[Proof of Theorem \ref{2.6}]
	By \cite[Corollary 1.9(ii)]{W1} (and its proof) the augmentation ideal $\ell_0^1(G, \omega)$ is generated by finitely many elements of the form $\delta_e - \delta_t \ (t \in G)$. We shall show that this is the only proper, finite-codimension left ideal. Indeed, let $I$ be a proper, left ideal of $\ell^1(G, \omega)$ of finite codimension. Then $G$ acts linearly on the quotient $E :=\ell^1(G, \omega)/I$, and since $G$ is simple the action is either trivial or faithful. If the action were faithful, then $G$ would be a linear group, which is impossible by Malcev's Theorem \cite{M40}. As such the action is trivial, and so $ \delta_t \cdot x = x$ for all $x \in E$ and all $t \in G$.  It follows that $(\delta_e - \delta_t) \cdot x = 0 \ (x \in E)$ for all $t \in G$ and, since these elements generate $\ell^1_0(G, \omega)$ as a left ideal, we must have $f \cdot x = 0$ for all $f \in \ell^1_0(G, \omega)$.
	As such $\ell^1_0(G, \omega)$ is contained in the kernel of the action of $\ell^1(G,\omega)$ on $E$, which is contained in $I$,
	and hence $I$ must equal the augmentation ideal by maximality. This completes the proof.
\end{proof}

	Finally, we turn towards the proof of Theorem \ref{2.7}. Our strategy for showing that the left ideals in question are finitely generated  will involve reducing to the case of an ideal of the form $J_\omega(G,H)$ for a finite-index subgroup $H$. We shall show that, for radial exponential weights, $J_\omega(\F_m, H)$ is finitely generated for any finite-index subgroup $H$ of $\F_m$. Then, exploiting the fact that any finitely-generated group is the quotient of some $\F_m$, we shall show that in fact $J_\omega(G, H)$ is finitely generated, whenever $G$ is a finitely-generated group, $H$ is a finite-index subgroup, and $\omega$ is a radial exponential weight. To ease the notation we fix the base of our radial exponential weights to be $2$, although this plays no significant role in the proofs.
	
	To carry this out we first require a technical lemma about free groups. Note that the subgroup $H$ of the lemma always has a generating set of the form described.

\begin{lemma}	\label{2.3}
	Let $m\in \N$, let $\F_m$ be the free group on the set $\{ a_1, \ldots, a_m\}$,
	and let $H$ be a finitely-generated subgroup of $\F_m$. Let $X = \{ a_1^{\pm 1}, \ldots, a_m^{\pm 1} \}$, and let $Y$ be a generating set for $H$ of the form $\dot{B}^X_r \cap H$, for some $r \in \N$. Take $u \in H$, and write $u = y_1 y_2 \cdots y_n$, where $n = |u|_Y$, and $y_1, \ldots, y_n \in Y$.  Write
	$$y_i = x_{i,1} x_{i,2}\cdots x_{i,k_i}  \qquad (i=1,\ldots, n)$$
	where $x_{i,p} \in X$ and $k_i = |y_i|_X$ , for $i=1,\ldots, n$, and $p = 1,\ldots, k_i$.
	\begin{enumerate}
		\item[\rm (i)] For each $i = 1, \ldots, n$, when we reduce 
		the expression for $y_iy_{i+1}$ given by 
		$$x_{i,1}\cdots x_{i, k_i}x_{i+1,1}\cdots x_{i+1, k_{i+1}}$$
		at most $\min \{ \lceil \frac{k_i}{2} \rceil , \lceil \frac{k_{i+1}}{2} \rceil  \} -1$ of the elements $x_{i,1}, \ldots, x_{i, k_i}$ can cancel. 
	\end{enumerate}	
\noindent
Moreover for each $j \in \{ 2, \ldots, n \}$ the following hold. 
		\begin{enumerate}
			\item[\rm (ii)] When we convert the expression for $y_1y_2 \cdots y_j$ given by
			$$x_{1,1}\cdots x_{1,k_1}x_{2,1} \cdots x_{2,k_2} \cdots x_{j,1}\cdots x_{j,k_j}$$
			 to a reduced word, the elements $x_{j,p}$ do not get canceled for $p= \lceil \frac{k_j}{2} \rceil, \ldots, k_j$.
			\item[\rm (iii)]  $ |y_1 y_2 \cdots y_{j-1}|_X < |y_1y_2\cdots y_{j-1}y_{j}|_X$.
		\end{enumerate}	
\end{lemma}

\begin{proof}
	We first prove part (i). Observe that if any more than $\lceil k_{i+1}/2 \rceil -1$ (that is to say, half or more) of the elements $x_{i+1,1}, \ldots, x_{i+1,k_{i+1}}$ are canceled when reducing 
	$x_{i,1}\cdots x_{i,k_1} x_{i+1,1}\cdots x_{i+1,k_2}$
	then we remove at least as many elements from $x_{i,1},\ldots, x_{i,k_i}$ as we add from $ x_{i+1,1},\ldots, x_{i+1,k_{i+1}}$, 
	so that $|y_iy_{i+1}|_X \leq |y_i|_X \leq r$. This implies that $y_iy_{i+1} \in Y\cup \{ e \}$ by our hypothesis on the special form of the generating set $Y$, and this contradicts $n = |u|_Y$. By a symmetrical argument we cannot cancel more than $\lceil k_i/2 \rceil -1$ elements either.
	
	We shall prove parts (ii) and (iii) together by induction on $j$. The base case $j=2$ follows from part (i).
	Suppose that the induction hypothesis holds for $j$. We shall now prove that (ii) holds for $j+1$. Write
	$$y_1y_2 \cdots y_j  = z_1 z_2 \cdots z_q,$$
	for some $z_1, \ldots, z_q \in X$, where $q = |y_1 \cdots y_j|_X$ and $ z_1 z_2 \cdots z_q$ is a reduced word over $X$.
	By the induction hypothesis 
	$$z_q = x_{j,k_j}, \quad z_{q-1} = x_{j, k_j-1}, \quad \ldots, \quad z_{q- \lfloor k_j/2 \rfloor} = x_{j, \lceil k_j/2 \rceil}.$$
	By part (i), when we reduce 
	$$y_jy_{j+1} = x_{j,1}\cdots x_{j,k_j}x_{j+1, 1} \cdots x_{j+1, k_{j+1}}$$ 
	at most $\lceil k_j/2 \rceil -1$ of the elements $x_{j+1, 1}, \ldots, x_{j+1, k_{j+1}}$ can cancel with 
	the elements $x_{j,1}, \ldots, x_{j,k_j}$. This means that when we reduce
	$$z_1 z_2 \cdots z_qx_{j+1, 1}\cdots x_{j+1, k_{j+1}}$$
	the only cancellations that can occur are between the elements $x_{j+1, 1}, \ldots, x_{j+1, k_{j+1}}$ and the elements $x_{j,1}, \ldots, x_{j,k_j}$. However, part (i) also implies that at most $\lceil k_{j+1}/2 \rceil -1$ of these elements can cancel, and so $x_{j+1, p}$ will not be canceled, for $p = \lceil k_{j+1}/2 \rceil, \ldots, k_{j+1}$ as required for part (ii). It also follows that we are canceling fewer terms from $z_1,\ldots, z_q$ than we are adding from $x_{j+1,1}, \ldots, x_{j+1, k_{j+1}}$, so that 
	$ |y_1y_2 \cdots y_{j-1}|_X < |y_1y_2\cdots y_{j-1}y_{j}|_X$, and part (iii) follows.
\end{proof}

\begin{lemma}	\label{2.4}
	Let $m \in \N$, write $\F_m = \langle a_1,\ldots, a_m \rangle$,  and let $X = \{ a_1^{\pm 1}, \ldots, a_m^{\pm 1} \}$. Set $\omega(t) = 2^{|t|_X} \ (t \in \F_m)$. Let $H$ be a finite-index subgroup of $\F_m$. Then the left ideal $J_\omega(\F_m, H)$ of $\ell^1(\F_m, \omega)$ is generated by finitely many elements of the form $\delta_e - \delta_t$, where $t \in H$.
\end{lemma}

\begin{proof}
	Let $k = [\F_m : H]$ and let $t_1, \ldots, t_k$ be a right transversal for $H$ in $\F_m$. Write $\gamma = \omega|_H$. Then $J_\omega(\F_m, H)$ can be written as 
	$$J_\omega(\F_m, H) = \ell^1_0(H,\gamma)*\delta_{t_1} + \cdots + \ell^1_0(H, \gamma)*\delta_{t_k}.$$
	As such, it suffices the prove that $\ell^1_0(H,\gamma)$ is  finitely generated as a left ideal of $\ell^1(H, \gamma)$. The weight $\gamma$ may not be radial, so we cannot apply \cite[Theorem 1.8]{W1} directly, but instead we adapt its proof.
	
	Since $H$ has finite index it is finitely generated. Let $Y$ be a finite generating set for $H$ of the form $\dot{B}_r^X \cap H$ for some $r \in \N$. Let given $u \in H$, and write $u = y_1y_2\cdots y_n$, for $y_1, \ldots, y_n \in Y$, where $n = |u|_Y$. By \cite[Lemma 6.2(i)]{W1},  we can write  
	
	$$\delta_e - \delta_u = \sum_{y \in Y} g_y*(\delta_e - \delta_y),$$
	
	where each $g_y \in \C H$ and has the form
	\begin{equation}		\label{eq1}
	g_y = \sum_{j=0}^{n-1} h_y^{(j)},
	\end{equation}
	where each $h_y^{(j)}$ is either $0$ or $\delta_{y_1\cdots y_j}$ in the case that $j \neq 0$, and is either $0$ or $\delta_e$ in the case that $j=0$. (Note that, contrary to what is stated immediately before it, \cite[Lemma 6.2(i)]{W1} is a purely algebraic statement that is independent of the weight; we shall see that we are able to prove that part (ii) of that lemma also holds for $\gamma$ by applying Lemma \ref{2.3} above). 
	
	For $j =1, \ldots, n-1$ we have $\|h_y^{(j)}\|_\gamma \leq 2^{|y_1 \cdots y_j|_X}$, and so  $\|g_y\|_\gamma \leq 1 + \sum_{j=1}^{n-1} 2^{|y_1 \cdots y_j|_X}$. By Lemma \ref{2.3}(iii) this sum consists of strictly increasing powers of $2$, each less than $2^{|u|_X}$ so that, 
	\begin{equation} 	\label{eq2}
	\|g_y\|_\gamma \leq \sum_{j=0}^{|u|_X -1} 2^j \leq 2^{|u|_X} = \gamma(u) \qquad (y \in Y).
	\end{equation}
	
	We can now mimic the proof of \cite[Theorem 1.8]{W1}. Enumerate $H = \{ u_0 =e, u_1, u_2 \ldots \}$, and let $f = \sum_{i=0}^\infty \alpha_i \delta_{u_i} \in \ell_0^1(H,\gamma)$. For each $i \in \N$ write 
	$$\delta_e - \delta_{u_i} = \sum_{y \in Y} g_y^{(i)} * (\delta_e - \delta_y)$$
	 for some functions $g_y^{(i)} \in \C H$ as in \eqref{eq1}. Define functions $\phi_y \ (y \in Y)$ by
	 $$\phi_y = -\sum_{i=0}^\infty \alpha_i g_y^{(i)},$$ 
	 and note that the series converges because, by \eqref{eq2}, $\| g_y^{(i)} \|_\gamma \leq \gamma(u_i) \ (i \in \N)$.
	 
	 Then, since $\sum_{i=0}^\infty \alpha_i =0$, we have
	\begingroup
	\allowdisplaybreaks
	\begin{align*}
	 f &= \sum_{i=0}^\infty \alpha_i \delta_{u_i} - \left(\sum_{i=0}^\infty \alpha_i \right) \delta_e = - \sum_{i=0}^\infty \alpha_i(\delta_e - \delta_{u_i}) \\
	 &= -\sum_{i=1}^\infty \alpha_i \left( \sum_{y \in Y} g_y^{(i)} *(\delta_e - \delta_y) \right) 
	 = \sum_{y \in Y} \phi_y*(\delta_e - \delta_y).
	\end{align*}
	\endgroup
	As $f$ was arbitrary, it follow that $\ell^1_0(H,\gamma)$ is generated by the finite set $\{\delta_e - \delta_y : y \in Y \}$, which completes the proof. 	
\end{proof}

The next lemma is an interesting generalisation of \cite[Corollary 1.9(ii)]{W1}, which states that $\ell^1_0(G, \omega)$ is finitely generated for any radial exponential weight $\omega$ on $G$.

\begin{lemma}	\label{2.5}
	Let $G$ be a finitely-generated group with a finite, symmetric generating set $X$, and let $H$ be a finite-index subgroup. Let $\omega$ be the weight on $G$ defined by $\omega(t) = 2^{|t|_X} \ (t \in G)$. Then the left ideal $J_\omega(G,H)$ is finitely generated in $\ell^1(G, \omega)$. Moreover, the generators can be chosen to each have the form $\delta_e - \delta_t$ for some $t \in H$.
\end{lemma}

\begin{proof}
	Write $X = \{u_1^{\pm 1},\ldots, u_n^{ \pm 1} \}$, and let $\{v_1, \ldots, v_m \}$ be a finite generating set for $H$, where $n, m \in \N$. Writing $\F_n = \langle a_1, \ldots, a_n \rangle$, let $q \colon \F_n \to G$ denote the surjective homomorphism defined by $a_i \mapsto u_i \ (i=1, \ldots, n)$. Let $K = q^{-1}(H)$, which is a finite-index subgroup of $\F_n$. 
	
	Let $\beta$ denote the weight on $\F_n$ given by $\beta(t) = 2^{|t|_Z}$, where $Z = \{ a_1^{\pm 1}, \ldots, a_n^{\pm 1} \}$,  and let $\widetilde{\beta}$ be the weight induced on $G$ by regarding it as a quotient of $\F_n$ via $q$, as in \eqref{eq5}. We claim that $\widetilde{\beta} = \omega$. Indeed, note that 
	$$|t|_X = \inf \{ |s|_Z : q(s) = t \},$$
	and by \eqref{eq5} we have 
	$$\widetilde{\beta}(t) = \inf\{ \beta(s) : q(s) = t \} = \inf\{2^{|s|_Z} : q(s) = t \} = 2^{\inf \{ |s|_Z : q(s) = t \}} = \omega(t).$$
	
	The map $q$ extends to a bounded algebra homomorphism, which we also denote by $q$, 
	$$q \colon \ell^1(\F_n, \beta) \to \ell^1(G, \widetilde{\beta}) = \ell^1(G, \omega).$$
	It is routinely checked that $q(J_\beta(\F_n,K)) = J_\omega(G,H)$. By Lemma \ref{2.4} $J_\beta(\F_n,K)$ is finitely generated, and hence so is $J_\omega(G,H)$. The final statement also follows from Lemma \ref{2.4} since $q(\delta_e - \delta_t) = \delta_e - \delta_{q(t)} \ (t \in K)$.
\end{proof}


We now prove our final theorem.

\begin{proof}[Proof of Theorem \ref{2.7}]
	(i) Let $f \in \ell^1(G,\omega) \setminus \{ 0 \}$. We shall show that there is a finite-codimension maximal left ideal $I$ in $\ell^1(G,\omega)$ for which $f \notin I$. By multiplying on the left by $\delta_t$ for some $t \in G$, we may assume that $f(e) \neq 0$. Let $\eps = \frac{1}{2}|f(e)|$, and let $F$ be a finite subset of $G$ containing the identity for which
	$$\sum_{t \in G \setminus F} |f(t)|\omega(t) <\eps.$$
	Since $G$ is residually finite there exists $H \lhd G$ of finite index such that $H \cap F = \{ e \}$. 
	
	Let $q \colon \ell^1(G, \omega) \to \ell^1(G/H, \widetilde{\omega})$ denote the quotient map. Then  
	\begin{equation}		\label{eq3}
	q(f)(e) = \sum_{s\in H}  f(s),
	\end{equation}
	which is non-zero because 
	$$\sum_{s \in H, s \neq e} |f(s)| \leq \sum_{s\in G \setminus F} |f(s)| < |f(e)|,$$
	so that $f(e)$ can never cancel with the rest of the terms in the sum \eqref{eq3}. 
	Therefore in particular $q(f) \neq 0$.
	 Since $G/H$ is a finite group, $\ell^1(G/H, \omega)$ is equal to $\C( G/H)$, which is semisimple. Therefore there exists a maximal left ideal $\widetilde{I}$ in $\ell^1(G/H, \widetilde{\omega})$ such that $q(f) \notin \widetilde{I}$. 
	Let $I = q^{-1}(\widetilde{I})$. Then $I$ is a finite-codimension maximal left ideal of $\ell^1(G, \omega)$, and $f \notin I$.
	 
	 (ii) Let $I$ be a finite-codimension left ideal of $\ell^1(G,\omega)$. Then $G$ acts linearly on the quotient 
	 $E:= \ell^1(G,\omega)/I$, and since $G$ is not a linear group, the action cannot be faithful, so has a kernel, which we shall denote by $H \lhd G$. Since $G$ is just infinite, $H$ must have finite index in $G$.
	 
	 We claim that $I \supset J_\omega(G,H)$. 
	As the action of $H$ on $E$ is trivial, we must have $\delta_t \cdot x = x$, and hence $(\delta_e-\delta_t) \cdot x = 0$, for all $t \in G$ and $x \in E$. By Lemma \ref{2.5} $J_\omega(G,H)$ is generated as a left ideal by elements of the form $\delta_e-\delta_t \ (t \in H)$, so it follows that $f \cdot x = 0$ for all $f \in J_\omega(G,H)$, which implies that	$I \supset J_\omega(G,H)$.
	
	Finally, fix finitely many elements $f_1, \ldots, f_k \in I$ such that $I = J_\omega(G,H) + \spn\{f_1, \ldots, f_k\}$. Then a finite generating set for $J_\omega(G,H)$ together with $f_1, \ldots, f_k$ form a finite generating set for $I$.
\end{proof}	

\textit{Remark.}
Note that in Theorem \ref{2.7}(ii) the left ideals are not assumed to be closed. 
However, it follows from the above proof that all of the finite-codimension left ideals of the Banach algebra in Theorem \ref{2.7} actually are closed. Indeed, we have seen that an arbitrary finite-codimension left ideal $I$ is equal to $ J_\omega(G,H) + \spn\{f_1, \ldots, f_k\}$, for some finite-index subgroup $H$ and some $f_1, \ldots, f_k \in I$, and since $J_\omega(G,H)$ is closed, $I$ must be as well. Not every unital Banach algebra has the property that all of its finite-codimension left ideals are closed: for intance, add a unit to any of the many interesting examples in \cite{CDP}. This property can be rephrased as an automatic continuity property of $\ell^1(G,\omega)$, namely that every module map from  $\ell^1(G,\omega)$ to a finite-dimensional Banach left $\ell^1(G,\omega)$-module is automatically continuous.
\vskip 2mm

\subsection*{Acknowledgements}
I would like to thank Ian Short for his support and hospitality during my time so far at the Open University. I would also like to thank Ulrik Enstad for his useful comments on an earlier draft of the article, and Karthik Rajeev her help with the reference \cite{A06}. Finally, I would like to thank the anonymous referee for their careful reading of the article.

\subsection*{Competing Interests}
The author declares none.


\begin{thebibliography}{99}

\bibitem{A06} M.\ Alb{\'e}rt, \emph{Representing graphs by the non-commuting relation}, Publ. Math. Debrecen \textbf{69}/\textbf{3} (2006), 261--269.

\bibitem{BGS} L.\ Bartholdi, R.\ Grigorchuk, and Z.\ Sunik, \emph{Branch groups}, Handbook of algebra, Vol. 3, North-Holland, Amsterdam (2003), 989--1112.

\bibitem{BK} D.\ Blecher and T.\ Kania, \emph{Finite generation in C*-algebras and Hilbert C*-modules}, Studia Math. \textbf{224} (2014), 143--151. 

\bibitem{CDP} M.\ Cabrera Garc{\'i}a, H.\  G.\ Dales, and {\'A}.\ Rodr{\'i}guez Palacios, \textit{Maximal Left Ideals in Banach Algebras}, Bul. London Math. Soc. \textbf{52} (2020), 1--15.

\bibitem{DKKKL} H.\  G.\ Dales, T.\ Kania, T.\ Kochanek, P.\ Koszmider and N.\ J.\ Laustsen, \emph{Maximal left ideals in the Banach algebra of operators on a Banach space}, Studia Math. \textbf{218} (2013), 245--286.  
	
\bibitem{DZ}  H.\ G.\ Dales and W.\ \.Zelazko, \emph{Generators of maximal left ideals in Banach algebras}, Studia Math. \textbf{212} (2012), 173--193. 

\bibitem{FT} A.\ V.\  Ferreira and G.\  Tomassini, \emph{Finiteness properties of topological algebras}, Ann.\  Scuola Norm. Sup. Pisa  \textbf{5} (1978), 471--488.

\bibitem{M40} A.\ I.\ Malcev, \emph{On isomorphic matrix representations of infinite groups of matrices} (Russian), Mat. Sb. \textbf{8} (1940) 405--422 \& Amer. Math. Soc. Transl. \textbf{45} (1965), 1--18. 

\bibitem{RS} H.\ Reiter and J.\ D.\ Stegeman, \emph{Classical Harmonic Analysis and Locally Compact Groups}, London Math. Soc. Monogr. \textbf{22}, Oxford University Press, 2000.



\bibitem{W1} J.\ T.\ White, \emph{Finitely-generated left ideals in Banach algebras on groups and semigroups}, Studia Math. \textbf{239} (2017), 67--99.

\bibitem{W0} J.\ T.\ White, \emph{Banach algebras on groups and semigroups}, PhD thesis, University of Lancaster (2018).
\end{thebibliography}
\end{document}